\documentclass[reqno,11pt]{amsart}
\usepackage{amsthm,amsmath,amsfonts,amssymb,amscd,mathrsfs,graphics,diagrams}
\usepackage{txfonts}%,pxfonts}
\usepackage{supertabular}
\usepackage[pdftex, 
			breaklinks=true,
            pdfauthor={Fan, Jarvis, Ruan},
            pdftitle={FJRW A_{r-1} is r-spin},
            pdfsubject={FJRW theory, Mirror symmetry},
pdfkeywords={FJRW, Mirror symmetry, r-spin,Witten}]{hyperref}
\newcommand{\tvp}{\widetilde{\varphi}}
\newcommand{\tcF}{\widetilde{\cF}}
\newcommand{\tC}{\widetilde{C}}

\newcommand{\BG}{\B G} %classifying stack of G
\newcommand{\B}{{\mathscr B}}

\newcommand{\C}{{\mathbb{C}}}

\newcommand{\E}{{\mathscr E}}

\newcommand{\LL}{{\mathscr L}} %an orbi-linebundle

\newcommand{\MM}{\overline{\M}} %the moduli space
\newcommand{\M}{{\mathscr M}}

\newcommand{\Q}{{\mathbb{Q}}}

\newcommand{\W}{{\mathscr W}}
\newcommand{\Z}{{\mathbb{Z}}}
\newcommand{\aut}{\operatorname{Aut}}

\newcommand{\be}{\mathbf{a}} %%basis of A_r-1 state space
\newcommand{\bgamma}{\boldsymbol{\gamma}} %the complete "type" of the W-structure--the image of the total evaluation map.
\newcommand{\bga}{\bgamma}

\newcommand{\bm}{\mathbf{m}}
\newcommand{\ba}{\mathbf{a}}

\newcommand{\bone}{\boldsymbol{1}}

\newcommand{\cCb}{\overline{\mathscr C}} %an orbicurve obtained from \cC
\newcommand{\cC}{\mathscr C} %an orbicurve
\newcommand{\cc}{\mathcal{C}}  %%universal non-orbi curve?
\newcommand{\cE}{\E}
\newcommand{\cF}{\mathscr{F}}
\newcommand{\cL}{\LL}
\newcommand{\cM}{\mathscr{M}}

\newcommand{\cS}{\mathscr{S}}

\newcommand{\chat}{\hat{c}}
\newcommand{\ch}{{\mathscr A}} %State space of A_{r-1} theory

\newcommand{\cv}{\boldsymbol{c}^{1/r}}

\newcommand{\etar}{{\eta}^{JKV}} %%pairing of the r-spin theory
\newcommand{\fL}{\mathfrak{L}}
\newcommand{\fix}{\operatorname{Fix}}
\newcommand{\frkc}{\mathfrak{C}} % used to indicate the entire data of a W-curve.

\newcommand{\ga}{{\gamma}}
\newcommand{\Ga}{{\Gamma}}
\newcommand{\hGamma}{\widehat{\Gamma}}

\newcommand{\irightarrow}{\widetilde{\longrightarrow}}

\newcommand{\lgr}[1]{\langle\gamma_{#1}\rangle}
\renewcommand{\r}{{\boldsymbol{s}}}  %%the basis of the r-spin space
\newcommand{\chr}{\mathscr{S}}  %%the r-spin state space

\newcommand{\st}{st} %stabilization
\newcommand{\tGamma}{\widetilde{\Gamma}}
\newcommand{\tcC}{\widetilde{\cC}}

\newcommand{\vp}{{\varphi}}

\renewcommand{\O}{{\mathscr{O}}}

\newcommand{\dsand}{\quad \text{ and } \quad} %insert an "and" in the middle of a displayed equation 
\newtheorem{thm}{Theorem}[subsection]

\newtheorem{prop}[thm]{Proposition}

\theoremstyle{definition}
\newtheorem{rem}[thm]{Remark}

\newtheorem{df}[thm]{Definition}
\newtheorem{ex}[thm]{Example}
\theoremstyle{remark}

\begin{document}

\date{\today}

\title[$A_{r-1}$ and $r$-Spin]{Quantum Singularity Theory for $A_{r-1}$ and  $r$-Spin Theory}
\author{Huijun Fan}
\thanks{Partially Supported by NSFC 10401001, NSFC 10321001, and NSFC 10631050}
\address{School of Mathematical Sciences, Peking University, Beijing 100871, China}
\email{fanhj@math.pku.edu.cn}
\author{Tyler Jarvis}
\thanks{Partially supported by National Science Foundation grant DMS-0605155 and NSA grant  \#H98230-10-1-0181}
\address{Department of Mathematics, Brigham Young University, Provo, UT 84602, USA}
\email{jarvis@math.byu.edu}
\author{Yongbin Ruan}
\thanks{Partially supported by the National Science Foundation and
the Yangtz Center of Mathematics at Sichuan University}
\address{Department of Mathematics, University of Michigan Ann Arbor, MI 48105 U.S.A
and the Yangtz Center of Mathematics at Sichuan University, Chengdu,
China} \email{ruan@umich.edu}

\maketitle

\begin{abstract}
We give a review of the theories of \cite{FJR} and the theory of \cite{JKV} and prove that for the singularity $A_{r-1} = x^{r}$ (with the group $G= \mu_r$), 
the stack of $A_{r-1}$-curves of \cite{FJR} is canonically isomorphic to the stack of $r$-spin curves described in \cite{AJ} and \cite{J}.  We further prove that the theory of \cite{FJR} satisfies all the axioms of \cite{JKV} for an $r$-spin virtual class.  Therefore, the results of \cite{JKV,Lee,FSZ,Giv} all apply to the $A_{r-1}$-theory of \cite{FJR}.  In particular, this shows that the Witten Integrable Hierarchies Conjecture is true for the $A_{r-1}$-theory of \cite{FJR}; that is, the total descendant potential function of the $A_{r-1}$-theory satisfies the $r$-th Gelfand-Dikii hierarchy.
\end{abstract}

\setcounter{tocdepth}{1}
\tableofcontents

\section{Introduction}
In the paper \cite{FJR} we introduce a family of moduli spaces, a virtual cycle,
and a corresponding cohomological field theory associated to each non-degenerate, quasi-homogeneous hypersurface singularity $W\in \C[x_1,\dots,x_N]$ and for each admissible subgroup $G$ of the diagonal automorphism group $G_{max} :=\{(\ga_1,\dots,\ga_N)\in (\C^*)^N | W(\ga_1 x_1,\dots,\ga_Nx_N) = W(x_1,\dots,x_N)\}$.

When the singularity is $A_{r-1} = x^{r}$, then the only admissible subgroup is the full automorphism group  $G_{max} = \mu_r \subset \C^*$, which is the group of $r$th roots of unity.

In this paper we prove that for the singularity $A_{r-1} = x^{r}$ (with the group $G= \mu_r$), 
the stack of $A_{r-1}$-curves of \cite{FJR} is canonically isomorphic to the stack of $r$-spin curves described in \cite{AJ} and \cite{J}, and that the theory of \cite{FJR} satisfies all the axioms of \cite{JKV} for an $r$-spin virtual class.  Therefore, the results of \cite{JKV,Lee,FSZ,Giv} all apply to the $A_{r-1}$-theory of \cite{FJR}.  In particular, this shows that the Witten Integrable Hierarchies Conjecture is true for the $A_{r-1}$-theory of \cite{FJR}; that is, the total descendant potential function of the $A_{r-1}$-theory satisfies the $r$-th Gelfand-Dikii hierarchy.

\subsection{Conventions and Notation}
Throughout this paper we will assume that $r>1$ is an integer and we will set $W:=x^r$.  The group of automorphisms of $W$ is $G = \mu_r = \{J^k| k\in \Z/r,\ J=\exp(2\pi i /r)\}$.  Although some of the results described in this paper hold in a more general setting, we will always work over $\C$.

\section{The stacks}

In this section we review the definition and basic properties of the stack of $r$-spin curves used in \cite{JKV,J} and of the stack of $A_{r-1}$-curves of \cite{FJR}.  These two stacks are isomorphic, as proved in \cite{AJ}.  We will briefly review that isomorphism in this section as well.

The stack of $A_{r-1}$ curves has a much simpler definition, and we will use that definition rather than the more complicated $r$-spin curve definition whenever possible.

\subsection{The stack of $A_{r-1}$-curves}

The definition in \cite{FJR} of an $n$-pointed $W$-curve $(\cC,\LL, p_1,\dots,p_n, \varphi)$ is an $n$-pointed stable orbicurve $(\cC,p_1,\dots, p_n)$, an orbifold line bundle $\LL$,  and an isomorphism $\varphi:\LL^r\rTo \omega_{\log}$, where $\omega_{\log}$ is the \emph{log-canonical bundle}---specifically, the bundle of meromorphic $1$-forms having at worst a single pole at each mark $p_i$ for $i\in \{1,\dots, n\}$.  That is, the sheaf of sections of $\omega_{\log}$ is locally generated by the one-form $dz/z$, where $z$ is a local coordinate on $\cC$ near a marked point $p_i$.  
We describe these structures in more detail below.

%%%\subsubsection?

Recall that an orbicurve $\cC$\glossary{CC@$\cC$ & An orbicurve} with
marked points $p_1, \dots, p_n$ is a (possibly nodal) Riemann surface
$C$\glossary{CCC@$C$ & The coarse (underlying) curve of the orbicurve
  $\cC$} with orbifold structure at each $p_i$ and each node.  That is
to say, for each marked point $p_i$ there is a local group $G_{p_i}$
and (since we are working over $\C$) a canonical isomorphism $G_{p_i}
\cong \Z/m_i$ for some positive integer $m_i$.  A neighborhood of
$p_i$ is uniformized by the branched covering map $z \rTo z^{m_i}$.
For each node $p$ there is again a local group $G_p \cong \Z/n_j$
whose action is complementary on the two different branches.  That is
to say, a neighborhood of a nodal point (viewed as a neighborhood of
the origin of $\{z w=0\}\subset \C^2$) is uniformized by a branched
covering map $(z,w)\rTo (z^{n_j}, w^{n_j})$, with $n_j\geq 1$, and
with group action $e^{2 \pi i /n_j}(z,w)=(e^{2 \pi i /n_j}z, e^{-2\pi
  i/n_j}w)$.

\begin{df}
We call the orbicurve $\cC$ \emph{smooth} if the underlying curve
$C$ is smooth, and we will call the orbicurve \emph{nodal} if the
underlying curve $C$ is nodal.
\end{df}
We denote by $\varrho:\cC \rTo C$\glossary{rho@$\varrho$ & The
  projection to coarse (underlying) space, given by forgetting the
  orbifold structure} the natural projection to the underlying
(coarse, or non-orbifold) Riemann surface $C$. If $\LL$ is a line
bundle on $C$, it can be pulled back to an orbifold line bundle
$\varrho^*{\LL}$ over $\cC$. When there is no danger of confusion, we
use the same symbol $\LL$ to denote its pullback.

\begin{df}\label{df:K-log}
Let $\omega_C$
be the canonical bundle of $C$.  We define the \emph{log-canonical
  bundle of $C$} to be the line bundle
 $$\omega_{C,\log} := \omega_C \otimes \O(p_1) \otimes \dots \otimes
  \O(p_n),$$ where
  $\O(p)$ is the holomorphic line bundle of degree one whose sections
  may have a simple pole at $p$.  This bundle $\omega_{C,\log}$ can be
  thought of as the canonical bundle of the punctured Riemann surface
  $C-\{p_1, \dots, p_n\}$.

The \emph{log-canonical bundle of $\cC$} is defined to be the pullback
to $\cC$ of the log-canonical bundle of $C$:
\begin{equation}\label{eq:Klog}
\omega_{\cC,\log} := \varrho^* \omega_{C,\log}.
\end{equation}
\end{df}

Near a marked point $p$ of $C$ with local coordinate $x$, the bundle
$\omega_{C,\log}$ is locally generated by the meromorphic one-form $dx/x$.
If the local coordinate near $p$ on $\cC$ is $z$, with $z^m=x$, then
the lift $\omega_{\cC,\log} :=\varrho^*(\omega_{C,\log})$ is still locally
generated by $m\, dz/z = dx/x$.  When there is no risk of confusion,
we will denote both $\omega_{C,\log}$ and $\omega_{\cC,\log}$ by $\omega_{\log}$.  Near
a node with coordinates $z$ and $w$ satisfying $zw=0$, both $K$
and $\omega_{\log}$ are locally generated by the one-form $dz/z = -dw/w$.

\subsubsection{Pushforward to the underlying curve}
\label{rem:desingularize}
If $\LL$ is an orbifold line bundle on a smooth orbicurve $\cC$, then
the sheaf of locally invariant sections of $\LL$ is locally free of
rank one, and hence dual to a unique line bundle $|\LL|$ on $\cC$. We
also denote $|\LL|$ by $\varrho_* \LL$, and it is called the
``desingularization" of $\LL$\glossary{Lbar@$"|\LL"|$ & The
  desingularization of the line bundle $\LL$} in \cite[Prop
4.1.2]{CR1}. It can be constructed explicitly as follows.

We keep the local trivialization at non-orbifold points, and change it
at each orbifold point $p$.  If $\LL$ has a local chart $\Delta \times
\C$ with coordinates $(z, s)$, and if the generator $1\in \Z/m \cong G_p$ acts locally
on $\LL$ by $(z,s) \mapsto (\exp(2\pi i/m) z, \exp(2\pi i v /m) s)$,
then we use the $\Z/m$-equivariant map $\Psi: (\Delta-\{0\})\times \C
\rTo \Delta \times \C$ given by
\begin{equation}(z, s)\rTo (z^m, z^{-v}s),
\label{eq:desing-triv}
\end{equation}
where $\Z/m$ acts trivially on the second $\Delta\times \C$.
Since $\Z/m$ acts trivially, this gives a line bundle over $C$, which is $|\LL|$.

If the orbicurve $\cC$ is nodal, then the pushforward $\varrho_*\LL$
of a line bundle $\LL$ may not be a line bundle on $C$.  In fact, if
the local group $G_p$ at a node acts non-trivially on $\LL$, then the
invariant sections of $\LL$ form a rank-one torsion-free sheaf on $C$
(see \cite{AJ}).  However, we may take the normalizations $\tcC$ and
$\widetilde{C}$ to get (possibly disconnected) smooth curves, and the
pushforward of $\LL$ from $\tcC$ will give a line bundle on
$\widetilde{C}$.  Thus $|\LL|$ is a line bundle away from the nodes of
$C$, but its fiber at a node can be two-dimensional; that is, there is
(usually) no gluing condition on $|\LL|$ at the nodal points.  The
situation is slightly more subtle than this (see \cite{AJ}), but for
our purposes, it will be enough to consider the pushforward $|\LL|$ as
a line bundle on the normalization $\widetilde{C}$ where the local
group acts trivially on $\LL$.

It is also important to understand more about the sections of the
pushforward $\varrho_*\LL$. Suppose that $s$ is a section of $|\LL|$
having local representative $g(u)$.  Then $(z, z^v g(z^m))$ is a
local section of $\LL$. Therefore, we obtain a section
$\varrho^*(s)\in \Omega^0(\LL)$ which equals $s$ away from orbifold
points under the identification given by
Equation~\ref{eq:desing-triv}. It is clear that if $s$ is holomorphic,
so is $\varrho^*(s)$. If we start from an analytic section of $\LL$,
we can reverse the above process to obtain a section of $|\LL|$. In
particular, $\LL$ and $|\LL|$ have isomorphic spaces of holomorphic
sections: $$\varrho^*: H^0(C,|\LL|) \irightarrow H^0(\cC,\LL). $$ In
the same way, there is a map $\varrho^*: \Omega^{0,1}(|\LL|)\rTo
\Omega^{0,1}(\LL)$, where $\Omega^{0,1}(\LL)$ is the space of orbifold
$(0,1)$-forms with values in $\LL$. Suppose that $g(u)d\bar{u}$ is a
local representative of a section of $t\in \Omega^{0,1}(|\LL|)$. Then
$\varrho^*(t)$ has a local representative $z^v g(z^m) m \bar{z}^{m-1}
d\bar{z}$.  Moreover, $\varrho$ induces an isomorphism
$$\varrho^* :H^1(C,|\LL|)\irightarrow H^1(\cC,\LL).$$

\begin{ex}
The pushforward $|\omega_{\cC,log}|$ of the log-canonical bundle of any
orbicurve $\cC$ is again the log-canonical bundle of $C$, because at a
point $p$ with local group $G_p \cong \Z/m$ the one-form $m\, dz/z =
dx/x$ is invariant under the local group action.

Similarly, the pushforward $|\omega_{\cC}|$ of the canonical bundle of
$\cC$ is just the canonical bundle of $C$:
\begin{equation}\label{eq:KbarIsK}
|\omega_{\cC}| = \varrho_*\omega_{\cC} = \omega_C,
\end{equation}
 because the local group $\Z/m$ acts on the one-form $dz$ by $\exp(2
 \pi i /m) dz$, and the invariant holomorphic one-forms are precisely
 those generated by $mz^{m-1}dz = dx$.
\end{ex}

%%%end saved part for later?

\begin{df}
An $A_{r-1}$-structure on an orbicurve $\cC$ is a choice of a line bundle $\LL$ and an isomorphism of line bundles
$$\varphi:\LL \rTo \omega_{\log},$$
with the additional condition that for each point $p \in \cC$, the 
 induced representation $\rho_p:G_p \rTo \aut(\LL)\cong U(1)$ be faithful.
\end{df}

An \emph{isomorphism of $A_{r-1}$-structures} $\Upsilon:(\LL,\varphi)\rTo (\LL',\varphi')$ on $\cC$ is defined to be an isomorphism $\xi:\LL \rTo \LL'$ such that $\varphi = \varphi' \circ  \xi$.

Different choices of maps $\varphi$ give isomorphic $A_{r-1}$-structures.
\begin{prop}
For a given orbicurve $\cC$,  any two $A_{r-1}$-structures $\fL_1:=(\LL,\vp)$ and $\fL_2:=(\LL,\vp')$ on $\cC$ which have identical structure bundle $\LL$ 
are isomorphic.
\end{prop}
\begin{proof}
The composition $\vp^{-1}\circ \vp'$ is an automorphism of $\omega_{\log}$ and hence defined by an element $\exp(\alpha) \in \C^*$.  Let $\beta = \alpha/r$.  This induces an automorphism $\exp(\beta):\LL\rTo\LL$ which takes $\vp$ to $\exp(\alpha)\vp = \vp'$, and thus induces an isomorphism of $A_{r-1}$-structures $\fL_1 \irightarrow \fL_2$.
\end{proof}

\begin{df} For each orbifold marked point $p_i$ we will denote the  image $\rho_{p_i}(1)$ of the canonical generator $1\in \Z/m \cong G_{p_i}$ in $U(1)$ by
$$\ga_i:=\ga_{p_i} :=\rho_{p_i}(1) = \exp(2\pi i \Theta^\ga).$$
\end{df}

The choices of orbifold structure for the line bundles in the $A_{r-1}$-structure is severely restricted by $W$; specifically, the faithful representation $\rho_{p_i}: G_{p_i} \rTo U(1)$ factors through $G$ so $\ga_i \in \mu_r$.
\begin{df}
A marked point $p$ of an $A_{r-1}$-curve is called \emph{narrow} (called Neveu-Schwarz in \cite{FJR})
 if $\ga_p = 1$.  The point $p$  is called \emph{broad} (called Ramond in \cite{FJR}) otherwise.
\end{df}

\subsection{Stack of stable $A_{r-1}$-orbicurves}

\begin{df}\label{df:stable-W-curve}
A pair $\frkc=(\cC,\fL)$\glossary{Cfrak@$\frkc$ & A $A_{r-1}$-curve $(\cC,\fL)$} consisting of  an orbicurve $\cC$ with $n$ marked points and with
$A_{r-1}$-structure $\fL$
is called a \emph{stable
$A_{r-1}$-orbicurve}
if the underlying curve $C$ is a stable curve.
\end{df}
\begin{df}
A \emph{genus-$g$, stable $A_{r-1}$-orbicurve with $n$ marked points over a
base $T$} is given by a flat family of genus-$g$, $n$-pointed
orbicurves $\cC \rTo T$ with (gerbe) markings $\cS_i \subset \cC$
and sections $\sigma_i: T \rTo \cS_i$, and the data $(\LL,\vp)$.  The sections $\sigma_i$ are required to
induce isomorphisms between
$T$ and the coarse moduli of $\cS_i$ for $i\in\{1,\dots,n\}$.  The bundle $\LL$ is an  orbifold line bundle on $\cC$.  And $\vp : \LL^r \irightarrow
\omega_{\cC/T, log}$ is an isomorphism to the
 relative log-canonical bundle
which, together with the $\LL$, induces an
$A_{r-1}$-structure on every fiber $\cC_t$.
\end{df}

%\begin{df}
%A \emph{morphism of stable $A_{r-1}$-orbicurves} $(\cC/T, \cS_1, \dots,
%\cS_n,\LL,
%\vp)$
%and $(\cC'/T', \cS'_1, \dots, \cS'_n, \LL',\vp')$
% is a tuple of morphisms $(\tau,\mu,
%\alpha)$ such that the pair $(\tau,\mu)$ forms
%a morphism of pointed orbicurves:
%$$\begin{diagram}
%\cC&\rTo^{\mu}&\cC'\\
%\dTo & & \dTo\\
%T &\rTo^{\tau}&T'\\
%\end{diagram}$$
%and $\alpha :\LL \irightarrow \mu^* \LL'$ is an isomorphism of $A_{r-1}$-structures on $\cC$.
%\end{df}

\begin{df}
We denote the stack of stable $A_{r-1}$-orbicurves by
$\W_{g,n}(A_{r-1})$ or simply $\W_{g,n}$.  
\end{df}

Forgetting the $A_{r-1}$-structure and the orbifold structure gives a morphism $$\st:\W_{g,n
} \rTo \MM_{g,n}.%\glossary{st@$\st$ & The stabilization morphism, which forgets the $A_{r-1}$-structure}
$$  The morphism $\st$ plays a role similar to that played by the {stabilization} morphism of stable maps.  

\begin{thm}[\cite{AJ}]
The stack $\W_{g,n}$ is a smooth, compact orbifold (Deligne-Mumford stack) with projective coarse moduli.  In particular, the morphism $st:\W_{g,n} \rTo \MM_{g,n}$ is flat, proper and quasi-finite (but not representable).  \end{thm}

\subsubsection{Decomposition of $\W_{g,n}$ into components}
The orbifold structure, and the image $\ga_i = \rho_{p_i}(1)$ of the canonical generator $1\in\Z/m_i\cong G_{p_i}$ at each marked point $p_i$ is locally constant, and hence is constant for each component of $\W_{g,n}$. Therefore, we can use these decorations to
decompose the moduli space into components.

\begin{df}\label{df:type}
For any choice  $\bgamma:=(\gamma_1, \dots, \gamma_n) \in G^n$ we define $\W_{g,n}(\bgamma)\subseteq \W_{g,n}$ to be the open and closed substack with orbifold
decoration $\bgamma$.\glossary{MMgkwgamma@$\W_{g,n}(\bgamma)$ & The stack of stable $A_{r-1}$-curves of genus $g$, with $n$ marked points and type $\bgamma$.}
We call $\bgamma$ the \emph{type} of any $A_{r-1}$-orbicurve in
$\W_{g,n}(\bgamma)$.
\end{df}
We have the decomposition
$$\W_{g,n}=\sum_{\bgamma} \W_{g,n}(\bgamma).$$

The following proposition is proved in \cite{FJR}.
 \begin{prop}
A necessary and sufficient condition for $\W_{g,n}(\bgamma)$ to be non-empty is
\begin{equation}\label{eq:deg-sel-rule}
(2g - 2 + n)/r-\sum^n_{l=1}\Theta^{\gamma_l} \in \Z.
\end{equation}
\end{prop}

\begin{ex}

For three-pointed, genus-zero $A_{r-1}$-curves, the choice of orbifold
line bundle $\LL$ providing the $A_{r-1}$-structure is
unique, if it exists at all. Hence, if the selection rule is satisfied, $\W_{0,3}(\bgamma)$ is isomorphic to the classifying stack
$\B\mu_r := [pt/\mu_r]$.
\end{ex}

\subsection{The stack of $r$-spin curves}

\subsubsection{Smooth $r$-spin curves}\label{smooth}
Let $g$ and $n$ be  integers such
that $2g-2+n>0$. Let $\bm=(m_1,\ldots,m_n)$ be an $n$-tuple of
integers.  
A nonsingular $n$-pointed $r$-spin curve of genus $g$ and type
$\bm$ over a base $S$, denoted $(C\rTo S, s_i, \cL,
c)$, is the data of
\begin{enumerate}
\item a smooth,
$n$-pointed curve $(C\rTo S, s_i: S \rTo C)$ of genus $g$,
\item an invertible sheaf $\cL$ on $C$, and
\item an isomorphism $c:\cL^{\otimes r} \irightarrow
\omega_{C/S}(-\sum_{i=1}^n m_i S_i)$, where $S_i$ is the image of
$s_i$.
\end{enumerate}

%In simple terms, a nonsingular $r$-spin curve of type
%$\ba$ is a pointed curve with the choice of an $r$-th
%root of the sheaf of holomorphic one-forms which vanish to order
%$a_i$ along the $i$th section.  
%
%A {\em morphism} from one $r$-spin curve $(C'\rTo S', s'_i, \cL',
%c')$ to another $(C\rTo S, s_i, \cL, c)$ consists of {\em a fiber
%diagram} $(\phi, \alpha)$, i.e., an $S$-morphism $\phi:C' \rTo C$,
%inducing an isomorphism $C'\rTo C\times_SS'$, 
%$$\begin{diagram}
%C' & \rTo^{\phi}& C \\
%\dTo & \square &\dTo\\
%S' &\rTo& S
%\end{diagram}
%$$
%along with an isomorphism $ \alpha: \cL' \rTo \phi^*\cL$, such that
%$\phi^*c\circ\alpha^{\otimes r} =c'$:
%
%

The category $\cM_{g,n}^{1/r,\bm}$ of nonsingular,
$n$-pointed $r$-spin curves of genus $g$ and type
$\ba$, with morphisms given by fiber diagrams, is a
Deligne-Mumford stack with quasi-projective coarse moduli space. 
See \cite{AJ,J} for a
detailed proof. When $\bm$ is congruent to $\bm' \bmod r$, the two
stacks $\cM_{g,n}^{1/r,\bm}$ and $\cM_{g,n}^{1/r,\bm'}$ are
canonically isomorphic.  We denote by $\cM^{1/r}_{g,n}$ the
disjoint union $\displaystyle \coprod_{\substack{\bm\\ 0 \leq k_i <r}}
\cM^{1/r,\bm}_{g,n}.$

Note that if $(C\rTo S, s_i, \cL, c)$ is a nonsingular, $n$-pointed
$r$-spin curve of genus $g$ and type $\bm$,  then 
\begin{equation}\label{rspinselrule}
\deg \cL = (2g-2-\sum
m_i/r) \in \Z.
\end{equation}
Moreover, the stack $\M^{1/r,\bm}_{g,n}$ is non-empty if and only if $$\sum
m_i \equiv 2g-2 \bmod r.$$

\subsubsection{The stack of stable $r$-spin curves}\label{stable}
To compactify the stack $\M_{g,n}^{1/r, \bm}$ one must replace $r$th-root line bundles by rank-one torsion-free sheaves on nodal curves as described in \cite{AJ,J}.  Although the full definition of the stable $r$-spin curve includes additional data, for our purposes the most important structure is the choice of an $n$-pointed stable curve $(C, p_1,\dots,p_n)$, a rank-one, torsion-free sheaf $\cE$ on $C$, and a morphism $\varphi:\cE^{\otimes r} \rTo \omega(-\sum_{i=1}^n m_ip_i)$ such that away from the nodes of $C$ the morphism $\varphi$ is an isomorphism.  Additionally, we require that at each node $q$ of $C$ one of the following two conditions holds.  
\begin{enumerate}
\item The morphism $\varphi$ is an isomorphism (in which case $\cF$ is locally free at the node).  In this case we call the spin structure \emph{Ramond} at the node $q$. 
\item The sheaf $\cF$ is not locally free, but $\cF$ and the morphism $\varphi$ are induced from a line bundle $\tcF$ and an isomorphism $\tvp:\tcF^{\otimes r} \rTo \omega_{\tC}(-m_+ q_+ - m_-q_-) $ on the normalization $\nu:\tC \rTo C$ of $C$ at the node $q$.  Here the normalization $\tC$ has two smooth points $q_+$ and $q_-$ lying over the node $q$, and the integers $m_+$ and $m_-$ must  satisfy $m_+ + m_- \equiv r-2 \pmod r$.  We call these integers the \emph{type} of the $r$-spin structure at the node $q$.    The inclusion $\omega_{\tC} \rInto \omega_{\tC,log} = \nu^*\omega_C$ gives a morphism $\tcF^{\otimes{r}} \rTo \nu^*\omega_C$ which induces (by adjointness) the morphism $\varphi:\cF^{\otimes r} \rTo \omega_C$. In this case we say that the spin structure is \emph{Neveu-Schwarz} at the node $q$.
\end{enumerate}
Additional data and restrictions must also be placed on the torsion-free sheaves to ensure that the compactification is smooth.  These restrict the way the spin structures may vary in
families and it involves additional data in the form of
intermediate roots of $\omega$.  For more details on these structures, see \cite{AJ,J}.

The category $\overline\cM_{g,n}^{1/r,\ba}$ of
$n$-pointed, stable $r$-spin curves of genus $g$ and type
$\ba$, with morphisms given by fiber diagrams, is a
smooth, proper Deligne-Mumford stack with projective coarse moduli space.  The forgetful map 
$\MM_{g,n}^{1/r} \rTo \MM_{g,n}$ is quasi-finite and proper, but not representable \cite{J}.

%There is a natural morphism $\overline\cM_{g,n}^{1/r,\bk}\rTo
%\overline{\mathscr{Q}}_{g,n}^{1/r,\bk}$ obtained by forgetting
%$\cF_d$ for $d\neq 1,r$. Since $\cF_d$ is torsion free, any
%automorphism of this sheaf is determined by its value over the
%generic points of $C$, and therefore the morphism above is
%representable.

We must also consider decorated dual graphs.
\begin{df}
A \emph{stable decorated graph} is a stable graph  with a marking  of each half-edge by
an integer $m$ with  $-1 \le m<r$, such that for each edge $e$ the
marks $m^+$ and $m^-$ of the two half-edges of $e$ satisfy $$ m^+
+ m^- \equiv r-2 \kern-.75em\pmod r. $$

Given a stable $r$-spin curve $\mathfrak{C}$ of type $ \bm=(m_1,
\dots, m_n)$, the \emph{decorated dual graph} of $\mathfrak{C}$ is
the dual graph $\Gamma$ of the underlying curve $C$, with the
following additional markings.  The $i$-th tail is marked by
$m_i$, and each half-edge associated to a node of $C$ is marked by
the type ($m^+$ or $m^-$) of the $r$-spin structure along the
branch of the node associated to that half-edge if the node is Neveu-Schwarz. We mark the half edges with $r-1$ if the corresponding node is Ramond.

Given a decorated stable graph $\Ga$, we denote by $\MM_{\Ga}^{1/r}$ the locus of $r$-spin curves in $\MM_{g,n}^{1/r}$ with dual graph $\Ga$.
\end{df}

\subsubsection{Change of type and roots of other bundles}

In the definitions given above, one could also have replaced the bundle $\omega$ with another line bundle defined on every stable curve.  Specifically, let $\bone:=(1,1,\dots,1)\in \Z^n$ be the $n$-tuple of all ones.  It will be important in this paper to compare $r$-spin curves to the stack obtained by replacing $\omega$ by  $\omega_{\log}:=\omega(\sum S_i)$, or by $\omega(-r\ell\bone):= \omega(-\sum r \ell S_i)$, or by $\omega_{\log}(-r\ell\bone):=\omega(\sum (1-r \ell) S_i)$, for any choice of $\ell\in \Z$.  We denote the corresponding stacks by $\MM_{g,n,log}^{1/r,\bm}$, $\MM_{g,n}^{1/r,\bm+r\ell\bone}$, and $\MM_{g,n,log}^{1/r,\bm+r\ell\bone}$, respectively.  More precisely, points of the smooth locus $\cM_{g,n,log}^{1/r,\bm}$ are tuples $(C,\LL,c)$ such that $C$ is an $n$-pointed,smooth, genus-$g$ curve, $\LL$ is a line bundle on $C$, and $c:\LL^{\otimes r} \irightarrow \omega_{C,log}(-\sum m_i p_i) =  \omega_{C}(\sum (1-m_i) p_i)$ is an isomorphism of line bundles on $C$.
Points of  the smooth locus $\cM_{g,n}^{1/r,\bm+r\ell\bone}$ are tuples $(C,\LL,c)$ such that $C$ is an $n$-pointed,smooth, genus-$g$ curve, $\LL$ is a line bundle on $C$, and $c:\LL^{\otimes r} \irightarrow \omega_{C}(-\sum (m_i+r\ell) p_i)$ is an isomorphism of line bundles on $C$.  And points of the smooth locus  $\cM_{g,n,log}^{1/r,\bm+r\ell\bone}$ are tuples $(C,\LL,c)$ such that $C$ is an $n$-pointed,smooth, genus-$g$ curve, $\LL$ is a line bundle on $C$, and $c:\LL^{\otimes r} \irightarrow \omega_{C,log}(-\sum (m_i+r \ell) p_i) =  \omega_{C}(\sum (1-m_i-r\ell) p_i)$ is an isomorphism of line bundles on $C$.

The stack $\MM_{g,n}^{1/r,\bm}$ is canonically isomorphic to the stack $\MM_{g,n}^{1/r,\bm+r\ell\bone}$ for any $m\in \Z$.  Similarly, the stacks $\MM_{g,n,log}^{1/r,\bm}$ and $\MM_{g,n,log}^{1/r,\bm+r\ell\bone}$ are canonically isomorphic for any $\ell\in \Z$.  In both cases, when $C$ is smooth, the isomorphism sends $(C,\LL,\varphi)$ to $(C,\LL\otimes\omega^m,\vp\otimes I)$, where $I:(\omega^{\otimes \ell})^{\otimes r} \irightarrow \omega^{\otimes \ell r}$ is the obvious isomorphism.

Similarly, the stack $\MM_{g,n,log}^{1/r,\bm}$ is isomorphic to the stack  $\MM_{g,n}^{1/r,\bm-\bone}$.  On the smooth locus this is immediate from the definitions.

\subsection{The isomorphism between $\W_{g,n}$ and $\MM^{1/r}_{g,n}$}

In \cite[\S4]{AJ} it is shown that the stack $\W_{g,n}(A_{r-1})$ of stable $A_{r-1}$-curves (in \cite{AJ} this stack is denoted $\mathcal{B}_{g,n}(\mathbb{G}_m,\omega_{\log}^{1/r})$) is isomorphic to the stack $\MM_{g,n}^{1/r}$ of $r$-spin curves.  On the smooth locus the isomorphism is given simply by pushing forward the orbifold line bundle to the  underlying (coarse) curve.  Before describing the isomorphism, we will review some facts about pushforwards of $A_{r-1}$ structures to the underlying curve.

\subsubsection{Pushforward of $A_{r-1}$-structures}

We now briefly recall some facts about the behavior of $A_{r-1}$-structures when forgetting the orbifold structure at marked points, that is, when they are pushed down to the underlying (coarse) curve.

An $A_{r-1}$-structure consists of a line bundle $\LL$ with an isomorphism $\LL^{\otimes r}\cong \omega_{\log}$ such that near an orbifold point $p$ with local coordinate $z$ the canonical generator $1\in\Z/m \cong G_p$ of the local group $G_p$ acts on $\LL$ by $(z,s) \mapsto (\exp(2 \pi i/m) z, \exp(2\pi i (v/m)) s)$ for some $v \in \{0, \dots, m-1\}$.   Since $\omega_{\log}$ is invariant under the local action of $G_p$, we must have $rv=\ell m$ for some  $\ell\in \{0,\dots, r-1\}$, and $\frac{v}{m}=\frac{\ell}{r}$.  Denote the (invariant) local coordinate on the underlying curve $C$ by $u=z^m$.  Any section in $\sigma\in\Omega^0(|\LL|)$ must locally be of the form $\sigma=g(u)z^{v}s$, in order to be $\Z/m$-invariant.  So
$\sigma^r$ has local representative $z^{r v
} g^r(u)\frac{dz}{z}=u^{\ell}g^r(u) \frac{du}{mu}$. Hence, $\sigma^r\in
\Omega^0(\omega_{\log} \otimes \O ((-\ell) p)$, and thus when $\ell \neq 0$, we have $\sigma^r\in \Omega^0(K)$.

%\begin{rem}\label{rem:zeros-of-powers}
%More generally, if $\LL^r \cong \omega_{\log}$ on a \emph{smooth} orbicurve with
%action of the local group on $L$ defined by $\ell_i >0$ (as above) at
%each marked point $p_i$, then we have $$(\varrho_* \LL)^r = |\LL|^r = \omega_{\log} \otimes \left(\bigotimes_i \O((-\ell_i) p_i)\right).$$ \end{rem}

From this we get the following proposition.
\begin{prop}\label{prop:pushforward-W}
%Let $(\LL,\vp)$ be an
%$A_{r-1}$-structure on an orbicurve $\cC$ which is smooth (the
%underlying curve $C$ is nonsingular) at an orbifold point $p \in
%\cC$.  Suppose also that the local group $G_p\cong \Z/m$ of $p$
%acts on $\LL$ by
%$$\gamma =(\exp (2 \pi i\Theta^\gamma);$$
%that is,
%$\exp(2 \pi i/m)(z,w)=(\exp(2 \pi i/m)z, \exp(2
%\pi i \Theta^\gamma) w)$ with $1> \Theta^\gamma \ge0$.
%
%Let $\cCb$ denote the orbicurve obtained from $\cC$ by making the orbifold structure at $p$ trivial (but retaining the orbifold structure at all other points).  Let $\varrhob:\cC \rTo \cCb$ be the obvious induced morphism, and let $\varrhob_*(\LL)$ denote the pushforward to $\cCb$ of an orbifold line bundle $\LL$ on $\cC$.
%
%Any isomorphism  $\psi:\LL^r \rTo \omega_{\log}$ induces an isomorphism on the pushforward
%$$\varrhob_*(\LL)^r \rTo  \omega_{\cCb, \log}\otimes
%\O\left(-r\Theta^\gamma
%p\right).
%$$

If $\cC$ is a smooth orbicurve (i.e., $C$ is a smooth curve), let
$\gamma_i$ define the action of the local group $G_{p_i}$ near $p_i$. The isomorphism $\varphi:\LL^r\rTo \omega_{\log}$ induces an isomorphism
\begin{equation}\label{eq:desing-bundle}
|\varphi|:|\LL|^{r}\rTo  \omega_{\cCb, \log}\otimes
\O\left(-\sum_{i=1}^n r\Theta^{\gamma_i}
p_i\right).
\end{equation}
\end{prop}

\subsubsection{The isomorphism between $\W_{g,n}$ and $\MM^{1/r}_{g,n}$}

On the smooth locus, the isomorphism between $\W_{g,n}$ and $\MM^{1/r}_{g,n}$ is a straightforward consequence of Proposition~\ref{prop:pushforward-W}.
Specifically, for any smooth $A_{r-1}$-curve $(\cC,\LL,\varphi)$ in $\W_{g,n}(\bga)$, the pushforward along $\varrho:\cC\rTo C$ to the underlying coarse curve induces a line bundle $|\LL|$ on $C$ and an isomorphism 
$$|\varphi|:|\LL|^r \rTo \omega_{\log}\left(\sum_{i=1}^n  -r\Theta^{\ga_i} p_i\right) \cong \omega\left(\sum_{i=1}^n  (1-r\Theta^{\ga_i}) p_i\right).$$
This describes the isomorphisms
$$\W_{g,n}(\bga) \rTo \MM_{g,n,log}^{1/r, (r\Theta^{\ga_1}, \dots,r\Theta^{\ga_n})} \rTo \MM_{g,n}^{1/r, (r\Theta^{\ga_1}, \dots,r\Theta^{\ga_n})-\bone}$$ 
on the smooth locus.

The details of the isomorphism on  nodal curves are more messy.  These details are not essential for our purposes here, but the interested reader may find them in \cite[\S4]{AJ}.

\section{The state spaces}

\subsection{The  $A_{r-1}$ state space of \cite{FJR}}

We now describe the state space $\ch:=\ch_{A_{r-1},\langle J \rangle}$ of the theory of \cite{FJR} for the singularity $A_{r-1}$ with symmetry group $\mu_r=\langle J \rangle$,  where $J = \exp(2 \pi i/r)$.  

\begin{df}
For each $\ga\in G\subset (\C^*)^N$ we define $\fix(\ga)\subseteq \C^N$ to the be fixed locus of $\ga$ acting on $C^N$. We also define $N_\ga$ to be the dimension of $\fix(\ga)$ (as a $\C$-vector space).
\end{df}

As described in \cite{FJR}, the state space is a direct sum of the $\langle J\rangle$-invariant part of certain middle-dimensional relative cohomology groups:
$$\ch_W = \bigoplus_{\ga\in G} \ch_{\ga} :=\bigoplus_{k=0}^{r-1} \ch_{J^k} = \bigoplus_{k=0}^{r-1} \left(H^{{N_{\ga}}}(\fix({J^{k}}),W^{\infty}_{J^{k}}, \Q)\right)^{\langle J \rangle}.$$
In the case that $k\neq 0$, the fixed locus of $J^k$ is $\{0\} \subset \Q$ so the middle cohomology for these elements is just $$H^{{N_{\ga}}}(\fix({J^{k}}),W^{\infty}_{J^{k}}, \Q) = H^{0}(\{0\},W^{\infty}_{J^{k}}, \Q)  \cong \Q,$$ and the $J$-action on $H^{0}(\{0\},W^{\infty}_{J^{k}}, \Q)$ is trivial. 

On the other hand, when $k=0$ we have $H^{{N_{\ga}}}(\fix({J^{0}}),W^{\infty}_{J^{k}}, \Q) = H^{1}(\Q,W^{\infty}_{J^0}, \Q)$, and a theorem of Wall \cite{Wall} shows that this is isomorphic, as a graded $G$-module, to the following space of germs of one-forms on $\C$:  
\begin{equation*}
H^{1}(\fix({J^{0}}),W^{\infty}_{J^{0}}, \Q) \cong \Omega^{1}/(x^{r-1}dx).
\end{equation*}
This space has the following simple basis:
$$
\{ dx, xdx, \dots, x^{r-2}dx\}.
$$

None of the elements in the $J^0$ sector are $\langle J \rangle$-invariant, and all the elements in the $J^{k}$ sectors are invariant when $k\neq0$.
Thus we have
$$\ch_{J^k} = \C\be_{J^k} \dsand 
\ch= \bigoplus_{k=1}^{r-1} \C\be_{J^{k}} ,$$ where $\be_{J^{k}} :=1 \in H^{N_{J^{k}}}(\C^N_{J^{-i}},    W^{\infty}_{J^{-i}}, \Q).$

In the notation of the previous section, the group element $J^k$ has $\Theta^{J^k} = k/r$.  Therefore the $J^k$-sectors correspond to marked points $p$ where the
$A_{r-1}$-structure gives $|\LL|^r \cong \omega_{\log} (-kp)$ near
$p$.  

When $k\neq 0$, we call these sectors \emph{narrow}.  These will correspond to the \emph{Neveu-Schwarz} sectors of $r$-spin theory.

For any non-degenerate $W$ and any admissible $G$, the state space
$\ch_{W,G}$ admits a
grading and a non-degenerate
pairing $\langle\,\rangle^W$.  The pairing in the $A_{r-1}$ case is given by 
$$\langle \be_{J^k},\be_{J^\ell}\rangle = \delta_{k+\ell \equiv 0 \kern-.75em\pmod r}.$$ 
The grading is more subtle, as we now describe.

\begin{df}The \emph{central
charge} of the singularity $W$ is denoted $\chat^{FJRW}$:
\begin{equation}
\chat^{FJRW}:=\sum_{i=1}^N
(1-2q_i).
\end{equation}
Suppose that $\gamma=(e^{2\pi
    i \Theta^\gamma_1}, \dots, e^{2\pi i\Theta^\gamma_{N}})$ for rational numbers
    $0\leq \Theta^\gamma_i<1$.
We define the \emph{degree shifting number}
\begin{equation}
 \iota_{\gamma}=\sum_i (\Theta^{\gamma}_i
 -q_i) \label{eq:degshift-thetaq}
 \end{equation}
 For a class $\alpha \in \ch_{\gamma}$, we define
\begin{equation}
\deg_W(\alpha)= \frac{N_\ga}{2}+ \iota_{\gamma}.
\end{equation}
      \end{df}
Specifically, in the $A_{r-1}$-case, we have 
$$\chat = \frac{r-2}{r},$$
and
$$\deg_{A_{r-1}}(\be_{J^k}) = \frac{k-1}{r},$$
provided we choose $0< k < r$.

%
%
%
%
%these sectors are called \emph{Neveu-Schwarz} in \cite{JKV,Witten}.
%The $J^0$-sectors also correspond to what are called \emph{Ramond
%  sectors} in \cite{JKV,Witten}, but the correspondence is slightly more
%subtle. The Ramond sectors of \cite{JKV} normally correspond to
%marked points $p$ where the $r$-spin structure bundle $L$ has $$L^r
%\cong K\otimes (-(r-1) p)$$ near $p$.  Letting $$L :=|\LL| \otimes
%\O(-p)$$ gives the desired correspondence (this might be
%considered a form of what was called \emph{descent} in \cite{JKV2}).
%
%

\subsection{The $r$-spin state space of \cite{JKV}}

The $r$-spin state space of \cite{JKV} is a $\C$-vector space 
$\chr$ of dimension $r-1$ with a
basis  $\{\r_0,\ldots, \r_{r-2}\}$ and metric $\etar$
given by
\begin{equation}
\etar(\r_\mu,\r_\nu):=\etar_{\mu\nu}:=%\frac{1}{r\lambda^2}
\delta_{\mu+\nu,r-2}.
\end{equation}
Furthermore, in \cite{JKV} the elements $\r_\mu$ are given a grading of 
$$\deg_{JVK}(\r_\mu) = \mu/r,$$ and the central charge of the theory is 
$$\chat^{JKV} = \frac{r-2}{r}.$$

\begin{rem}
In \cite{JKV,Wi} the span of each element $\r_\mu$ in $\chr$ is called a \emph{Neveu-Schwarz sector}.  There is an additional one-dimensional \emph{Ramond sector} that appears in the $r$-spin theory, but this sector completely decouples from the rest of the theory and so can be omitted (see \cite[Rem 3.10]{JKV} for more about this sector).  This Ramond sector can be thought of as roughly corresponding to the case of $\ga = J^0$ in the $A_{r-1}$ theory.  
\end{rem}

\subsection{The State Space Isomorphism} 

The two state spaces are isomorphic as graded vector spaces with metric.  
As suggested by the grading, the isomorphism matches the element $\be_{J^k} \in \ch_{J^k}$ to the element $\r_{k-1} \in \chr$.  It is straightforward to see that this is indeed an isomorphism preserving the grading, that $\deg_{FJRW}(\be_{J^k}) = (k-1)/r = \deg_{JKV} \r_{k-1}$, and that $\chat_{FJRW}  = \chat_{JKV} = (r-2)/r$.  From now on we will drop the subscripts and just write $\deg$ to denote both $\deg_{A_{r-1}}$ and $\deg_{JKV}$, write $\eta$ or $\langle\, , \rangle$ to denote the pairing, and write $\chat$ to denote $\chat_{FJRW} = \chat_{JKV} = (r-2)/r$.

\section{The virtual classes}
In this section we show how to construct an $r$-spin virtual cohomology class using the virtual class of the $A_{r-1}$ theory.  We begin by reviewing the axioms of the $r$-spin virtual cohomology class and the properties of the $A_{r-1}$ virtual cycle.

\subsection{The \cite{JKV} axioms of an $r$-spin virtual class}
The $r$-spin theory depends on the existence of a \emph{virtual cohomology class} $\cv$ satisfying a list of axioms outlined in \cite[\S4.1]{JKV}.  We briefly review those axioms here.

\begin{df}
An \emph{$r$-spin virtual class} is an assignment of a cohomology
class
\begin{equation}
\label{eq:cvirt} \cv_\Gamma \in H^{2 D} \left(\MM^{1/r}_\Gamma, \Q\right)
\end{equation}
 to every genus $g$, stable, decorated graph $\Gamma$ with
 $n$-tails.
Here, if the tails of $\Gamma$ are marked with the $n$-tuple
$\bm=(m_1,\ldots,m_n)$, then the dimension $D$ is
\begin{equation}
\label{eq:deg} D = \chat(g-\alpha)+\sum_{i=1}^n
m_i/r  = \chat(g-\alpha)+\sum_{i=1}^n \deg(\r_{m_i}),
\end{equation}
and $\alpha$ is the number of connected components of $\Gamma$. In
the special case where $\Gamma$ has one vertex and no edges, we
denote $\cv_\Gamma$ by $\cv_{g,n}(\bm)$.   These classes must
satisfy the axioms below.
\begin{description}
\item[\bf Axiom 1a (Connected Graphs)]
Let $\Gamma$ be a connected, genus $g$, stable, decorated graph
with $n$ tails. Let $E(\Gamma)$ denote the set of edges of
$\Gamma$. For each edge $e$ of $\Gamma$, let
$l_e\,:=\,\gcd(m_e^++1,r)$, where $m_e^+$ is an integer decorating
a half-edge of $e$. The classes $\cv_{\Gamma}$ and
$\cv_{g,n}(\bm)$ are related by
\begin{equation}
\label{eq:cvgamma} \cv_{\Gamma}= \left(
\prod_{e\,\in\,E(\Gamma)}\,\frac{r}{l_e}\right) \, \tilde{i}^*
\,\cv_{g,n}(\bm) \in H^{2 D} (\MM^{1/r}_{\Gamma}),
\end{equation}
where $\tilde{i}\, :\,\M_\Gamma^{1/r}\,\rInto\M_{g,n}^{1/r,\bm}$ is the
canonical
inclusion map.
\item[\bf Axiom 1b (Disconnected Graphs)]
Let $\Gamma$ be a stable, decorated graph which is the disjoint
union of connected graphs $\Gamma^{(d)}$, then the classes
$\cv_{\Gamma}$ and $\cv_{\Gamma^{(d)}}$ are related by
\[
\cv_\Gamma\,=\,\bigotimes_d\,\cv_{\Gamma^{(d)}}\,\in\,H^\bullet(\M^{1/r}_\Gamma).
\]
\smallskip

\item[\bf Axiom 2  (Convexity)]
If $m_i \neq r-1$ for all $i\in \{1,\dots,n\}$, let $\cF$ denote the universal $r$th root rank-one torsion-free sheaf on the universal $r$-spin curve 
$\pi:\cc^{1/r,\bm}_{g,n}\rTo
\MM^{1/r,\bm}_{g,n}$. For each irreducible (and connected)
component of $\M^{1/r,\bm}_{g,n}$ (denoted here by
$\M^{1/r,\bm,(d)}_{g,n}$  for some index $d$), if
$\pi_*\cF_r=0$ on $\M^{1/r,\bm,(d)}_{g,n}$, then
$\cv_{g,n}(\bm)$ restricted to
$\M^{1/r,\bm,(d)}_{g,n}$ is the top Chern class $(-1)^Dc_D (R^1\pi_*\cF_r)$,
of the dual of the first derived pushforward of $\cF$. 
\smallskip

\item[\bf Axiom 3 (Cutting edges)]
Given any genus-$g$, decorated stable graph $\Gamma$ with $n$ tails
marked with $\bm$, we have a diagram
\begin{equation}
 \begin{diagram} \label{eq:Cutting}
 & &
 \M_{\tilde{\Gamma}}
\times_{\M_{\Gamma}}
 \M^{1/r}_{\Gamma} & \rTo^{\tilde{\mu}} &
 \M^{1/r}_{\Gamma} &
 \rInto^{\tilde{i}} & \M^{1/r}_{g,n} \\
 & \ldTo_{p_1} \\
 \M^{1/r}_{\tilde{\Gamma}} & & \dTo^p & & \dTo^p & &
 \dTo^p \\
 & \rdTo_{p_2}\\
  & & \M_{\tilde{\Gamma}} & \rTo^{\mu} &
 \M_{\Gamma} &
 \rInto^i & \M_{g,n}.
 \end{diagram}
\end{equation}
where $\M_{\tilde{\Gamma}}$ is the stack of stable
 curves with
 graph $\tilde{\Gamma}$, the graph obtained by cutting all edges of
 $\Gamma$, and $\M^{1/r}_{\tilde{\Gamma}}$ is the stack of
 stable $r$-spin curves with graph $\tilde{\Gamma}$ (still marked
 with $m^{\pm}$ on each half edge).
 $p_1$ is the following morphism:
 The fiber product consists of triples of an $r$-spin curve $(X/T,
 \cF,\varphi)$, a stable curve $\tilde{X}/T$, and a
 morphism $\nu : \tilde{X} \rTo X$, making $\tilde{X}$ into the
 normalization of  $X$.  Also, the dual graphs of $X$ and $\tilde{X}$
 are $\Gamma$   and $\tilde{\Gamma}$, respectively.  The associated $r$-spin
 curve in $\M^{1/r}_{\tilde{\Gamma}}$ is simply
 $(\tilde{X}/T,   \nu^*\cF, \nu^*\varphi)$.
We require that
\[
p_{1*} \tilde{\mu}^* \cv_{\Gamma} =  r^{|E(\Gamma)|}
\cv_{\tilde{\Gamma}} ,\] where $E(\Gamma)$ is the set of edges
of $\Gamma$ that are cut in
 $\tilde{\Gamma}$.
\item[\bf Axiom 4 (Ramond Vanishing)]
\label{vanish}
  If $\Gamma$ contains a tail marked with $m_i=r-1$ or $m_i=-1$, then
$\cv_{\Gamma}=0.$
\item[\bf Axiom 5 (Forgetting tails)]
\label{forget} Let $\widehat{\Gamma}$ be a stable graph whose
$i$-th tail is marked by $m_i=0$, $\Gamma$ be the stable graph
obtained by removing the $i$-th tail, and
$$\pi:\MM^{1/r}_{\widehat{\Gamma}} \rTo\MM^{1/r}_{\Gamma} $$
be the forgetful morphism.  The classes $\cv_{\widehat{\Gamma}}$
and $\pi^*\cv_{\Gamma}$ are related by
$$\cv_{\widehat{\Gamma}}=\pi^*\cv_{\Gamma}.$$
\end{description}
\end{df}

\subsection{The properties satisfied by the \cite{FJR} virtual cycle for $A_{r-1}$}

Let $\Gamma$ be a stable graph (not necessarily
connected) with tails $T(\Ga)$, and with each tail $\tau\in T(\Ga)$ decorated by an element
$\gamma_\tau \in G$.  Denote by $n = |T(\Ga)|$ the number of
tails of $\Gamma$.

The theory of \cite{FJR} provides a homology \emph{virtual cycle} 
\begin{align*}\left[\W(\Gamma)\right]^{vir}  \in &
H_*(\W(\Gamma),\Q)\otimes \prod_{\tau \in T(\Gamma)}
\left(H_{{N_{\ga}}}(\fix({J^{k}}),W^{\infty}_{J^{k}}, \Q)\right)^{\langle J \rangle}\\
& = H_*(\W(\Gamma),\Q)\otimes \prod_{\tau \in T(\Gamma)}\ch^*_{\ga_{\tau}},
\end{align*}
where $\ch^*_{\ga_\tau}$ is the dual of $\ch_{\ga_\tau}$.  And the virtual cycle satisfies several axioms similar to those of the virtual cohomology class of $r$-spin theory.  Here we briefly review those properties of the cycle that are relevant to this paper.

When $\Gamma$ has a single vertex of genus $g$, $n$ tails, and no edges
(i.e, $\Gamma$ is a corolla), we denote the virtual cycle by
$\left[\W(\bgamma)\right]^{vir}$, where $\bgamma :=(\gamma_1,
\dots, \gamma_n)$.

The following properties axioms hold for the virtual cycle $\left[\W(\Gamma)\right]^{vir}$:
\begin{enumerate}
\item \textbf{Dimension:}\label{ax:dimension} As in the $r$-spin case, define $$D=\chat(\ga-\alpha) + \sum_{\tau\in T(\Ga)} \deg (\be_{\ga_\tau}),$$ where $\alpha$ is the number of connected components of $\Ga$.
If $D$ is not a
  half-integer\footnote{Note that in the $A_{r-1}$ case, the selection rule Equation~(\ref{eq:deg-sel-rule}) shows that the stack is empty unless $D$ is actually an integer.
} (i.e., if $D \not\in\frac{1}{2}\Z$), then
  $\left[\W(\Gamma)\right]^{vir}=0$. Otherwise, the cycle
  $\left[\W(\Gamma)\right]^{vir}$ has degree
\begin{equation}\label{eq:dimension}
2\left((\chat-3)(1-g) + n - \sum_{\tau\in
T(\Gamma)} \iota_{\tau}\right).
\end{equation}
So the cycle lies in $H_d(\W(\Gamma),\Q)\otimes \prod_{\tau \in
  T(\Gamma)}
\ch^*_{\ga_{\tau}},$
%H_{N_{\gamma_{\tau}}}(\C^N_{\gamma_{\tau}},W^{\infty}_{\gamma_\tau},\Q),$ 
where
$$
d:=6g-6+2n -2D  = 2\left((\hat{c}-3)(1-g)+n-\sum_{\tau\in T(\Gamma)}\deg(\be_{\ga_\tau})\right).
$$

\item \label{ax:symm}\textbf{Symmetric group invariance}: There is a
  natural $S_n$-action on $\W_{g,n}$ obtained by permuting the
  tails.  This action induces an action on homology.  That is, for any
  $\sigma \in S_n$ we have:
$$\sigma_*: H_*(\W_{g,n},\Q)\otimes \prod_i
\ch^*_{\ga_i} 
%H_{N_{\gamma_{i}}}(\C^N_{\gamma_i}, W^{\infty}_{\gamma_i}, \Q)^{G_W}
\rTo{}{}  H_*(\W_{g,k},\Q)\otimes \prod_i
\ch^*_{\ga_i} 
%H_{N_{\gamma_{\sigma(i)}}}(\C^N_{\gamma_{\sigma(i)}},W^{\infty}_{\gamma_{\sigma(i)}}, \Q)^{G_W}
.$$ For any decorated graph
$\Gamma$, let $\sigma\Gamma$ denote the graph obtained by applying
$\sigma$ to the  tails of $\Gamma$.

We have
\begin{equation}\sigma_*\left[\W(\Gamma)\right]^{vir} = \left[\W(\sigma\Gamma)\right]^{vir}.\end{equation}

 \item \textbf{Degenerating connected graphs:} \label{ax:ConnGraphs} Let $\Gamma$ be a
   connected, genus-$g$, stable graph decorated with $\ga_i$ on the $i$th tail.

The cycles $\left[\W(\Gamma)\right]^{vir}$ and
$\left[\W_{g,n}(\bgamma)\right]^{vir}$ are related by
\begin{equation}
 \left[\W(\Gamma)\right]^{vir}=
\tilde{i}^*\left[\W_{g,n}(\bgamma)\right]^{vir},
\end{equation}
where $\tilde{i} : \W(\Gamma) \rTo{}{} \W_{g,n}(\bgamma)$ is the
canonical inclusion map.
\item \textbf{Disconnected graphs:} Let
$\Gamma =\coprod_{i} \Gamma_i$ be a stable, decorated $W$-graph
which is the disjoint union of connected $W$-graphs $\Gamma_i$.
The classes $\left[\W(\Gamma)\right]^{vir}$ and
$\left[\W(\Gamma_i)\right]^{vir}$ are related by
\begin{equation}
\left[\W(\Gamma)\right]^{vir}=
\left[\W(\Gamma_1)\right]^{vir} \times \cdots \times
\left[\W(\Gamma_d)\right]^{vir}.
\end{equation}
\smallskip

\item{\bf Concavity}:\label{ax:convex}
%\footnote{This axiom was
%  called \emph{convexity} in \cite{JKV} because the original form of
%  the construction outlined by Witten in the $A_{r-1}$ case involved
%  the Serre dual of $\LL$, which is convex precisely when our $\LL$ is
%  concave.}
 Suppose that all the decorations on tails of $\Gamma$ are
  \emph{narrow}, meaning that $\fix(\ga) = \C^{N_{\gamma_i}}=\{0\}$, and so we can
  omit $\ch^*_{\ga_i} = H_{N_{\gamma_{i}}}(\C^{N_{\gamma_i}}, W^{\infty}_{\gamma_i}, \Q)\cong\Q$ from our notation.

Consider the universal $A_{r-1}$-structure  bundle $\LL$ on the
universal curve $\pi:\cC \rTo{}{} \W(\Gamma)$. 

If
$\pi_*(\LL)=0$, then the virtual
cycle is given by capping the top Chern class of the dual $\left(R^1 \pi_* (\LL)\right)^*$ of the pushforward
with the usual fundamental cycle of the moduli space:
\begin{equation}\begin{split}
\left[\W(\Gamma)\right]^{vir}& = c_{top}\left((R^1\pi_*\LL)^*\right) \cap \left[\W(\Gamma)\right]\\
& = (-1)^D c_{D}\left(R^1\pi_*\LL\right)
\cap \left[\W(\Gamma)\right].
\end{split}
\end{equation}

%\item   {\bf  Index zero:} \label{ax:wittenmap} Suppose that  $\dim (\W(\Gamma))=0$
%and all the decorations on tails  are narrow.
%
%If the pushforwards $\pi_* \left(\bigoplus\LL_i\right)$ and $R^1\pi_*
%\left(\bigoplus \LL_i\right)$ are both vector bundles of the same
%rank, then the virtual cycle is just the degree $\deg(\wit)$ of the
%Witten map times the fundamental cycle:
%$$\left[\W(\Gamma)\right]^{vir} = \deg(\wit)\left[\W(\Gamma)\right],$$

\item\textbf{Composition law:}\label{ax:cutting} Given any genus-$g$
  decorated stable $W$-graph $\Gamma$ with $k$ tails, and given any
  edge $e$ of $\Gamma$, let $\hGamma$ denote the graph obtained by
  ``cutting'' the edge $e$ and replacing it with two unjoined tails
  $\tau_+$ and $\tau_-$ decorated with $\gamma_+$ and $\gamma_-$,
  respectively.

The fiber product
$$F:=\W(\hGamma)\times_{\W(\Gamma)} \W(\Gamma)$$
has morphisms
 $$ \W({\hGamma})\lTo^{q} F
\rTo^{pr_2}\W(\Gamma).$$

We have
\begin{equation}\label{eq:cutting}
\left\langle \left[\W(\hGamma)\right]^{vir}\right\rangle_{\pm}=\frac{1}{\deg(q)}q_*pr_2^*\left(\left[\W(\Gamma)\right]^{vir}\right),
\end{equation}
where $\langle  \rangle_{\pm}$ is the map from
$$H_*(\W(\hGamma))\otimes\prod_{\tau \in T(\Gamma)} 
\ch^*_{\ga_\tau}%H_{N_{\gamma_{\tau}}}(\C^N_{\gamma_{\tau}},W^{\infty}_{\gamma_\tau}, \Q) 
\otimes  
\ch^*_{\ga_{+}}%H_{N_{\gamma_{+}}}(\C^N_{\gamma_{+}},W^{\infty}_{\gamma_+}, \Q)
\otimes 
\ch^*_{\ga_{-}}%H_{N_{\gamma_{-}}}(\C^N_{\gamma_{-}},W^{\infty}_{\gamma_-}, \Q)
$$ to
$$ H_*(\W(\hGamma))\otimes\prod_{\tau \in T(\Gamma)} 
\ch^*_{\ga_\tau}%H_{N_{\gamma_{\tau}}}(\C^{N_{\gamma_{\tau}}},W^{\infty}_{\gamma_\tau}, \Q)
$$ obtained by contracting  the last two factors via the usual dual pairing
$$\langle\, , \rangle: \ch^*_{\ga_{+}}%H_{N_{\gamma_{+}}}(\C^N_{\gamma_{+}},W^{\infty}_{\gamma_+}, \Q)
\otimes 
\ch^*_{\ga_{-}}%H_{N_{\gamma_{-}}}(\C^N_{\gamma_{-}},W^{\infty}_{\gamma_-}, \Q)
\rTo \Q.$$
\item
\textbf{Forgetting tails:}\label{ax:tails}
\begin{enumerate}
\item
 Let $\Gamma$ have its $i$th tail decorated with $J$, where $J$
 is the exponential grading element of $G$. Further let $\Gamma'$ be
 the decorated $W$-graph obtained from $\Gamma$ by forgetting the
 $i$th tail and its decoration.  Assume that $\Gamma'$ is stable, and
 denote the forgetting tails morphism by $$\vartheta: \W(\Gamma)
 \rTo{}{} \W(\Gamma').$$ We have
\begin{equation}
\left[\W(\Gamma)\right]^{vir}
=\vartheta^*\left[\W(\Gamma')\right]^{vir}.
\end{equation}
    \item In the case of $g=0$ and $k=3$, the space
      $\W(\gamma_1, \gamma_2, J)$ is empty if
      $\gamma_1\gamma_2\neq 1$, and otherwise $\W_{0,3}(\gamma, \gamma^{-1},
      J)=\BG$.  We omit
      $\ch^*_{J} = %H_{N_{J}}(\C^N_{J},W^{\infty}_{J}, \Q)^{G} = 
      \Q$
      from the notation.  In this case, the cycle
            $$\left[\W_{0,3}(\gamma, \gamma^{-1},
        J)\right]^{vir}\in H_*(\BG,\Q)\otimes
\ch^*_\ga %     H_{N_{\gamma}}(\C^N_{\gamma},W^{\infty}_{\gamma}, \Q)^{G_W} 
\otimes
\ch^*_{\ga^{-1}} %     H_{N_{\gamma^{-1}}}(\C^N_{\gamma^{-1}},W^{\infty}_{\gamma^{-1}},\Q)^{G_W}
$$ is the fundamental cycle of $\BG$ times the Casimir
      element. Here the Casimir element is defined as follows. Choose
      a basis $\{\alpha_i\}$ of
$\ch^*_\ga % H_{N_{\gamma}}(\C^N_{\gamma},W^{\infty}_{\gamma}, \Q)^{G_W}
,$ and a
      basis $\{ \beta_j\}$ of
$\ch^*_{\ga^{-1}}%H_{N_{\gamma^{-1}}}(\C^N_{\gamma^{-1}},W^{\infty}_{\gamma^{-1}},\Q)^{G_W}
$. Let $\eta_{ij}=\langle \alpha_i, \beta_j\rangle $ and
      $(\eta^{ij})$ be the inverse matrix of $(\eta_{ij})$. The
      Casimir element is defined as $\sum_{ij}\alpha_i\eta^{ij}\otimes
      \beta_j.$
\end{enumerate}

\end{enumerate}

The virtual cycle satisfies several other properties as well, but those additional properties are not needed for this paper.

\subsection{The $A_{r-1}$ virtual cycle defines an $r$-spin class}
As noted above, the stack of $r$-spin curves and the stack of $A_{r-1}$-curves are isomorphic as stacks, and the state spaces of the two theories are isomorphic.  All that remains is to use the $A_{r-1}$ virtual cycle to construct a cohomology class which satisfies the axioms of an $r$-spin virtual class.

\subsubsection{$r$-spin class from the $A_{r-1}$ virtual cycle}

To complete the connection to the $r$-spin theory of \cite{JKV}, we
define a cohomology class $c_{g,k}^{1/r}(\bm)$ as follows.
\begin{df}\label{df:Gammatilde}
Given a stable graph $\Gamma$ with tails (indexed by $i\in
\{1,\dots,n\}$) decorated by integers $m_i \in \{0,\dots,r-1\}$ and
half-edges (indexed by $e_+$ and $e_-$ for $e\in E(\Gamma)$) decorated
by integers $m_{e_+}$ and $m_{e_-}$ in $\{0,\dots,r-1\}$, such that
for any edge the two decorations $m_{e_+}$ and $m_{e_-}$ on the half
edges $e_+$ and $e_-$, respectively, satisfy the relation
$m_{e_+}+m_{e_-} \equiv r-2 \pmod r$, we let $ \tGamma$ be the stable
decorated $W$-graph whose tails are decorated with the group
elements $\gamma_i := J^{m_i+1}$ and whose half-edges are decorated
by the group elements $\gamma_{e_+} := J^{m_{e_+}+1}$ and
$\gamma_{e_-} := J^{m_{e_-}+1}$, respectively.  

We define $\cv_{\Gamma}$ to be the cohomology class
\begin{equation}
\cv_{\Gamma}:= \begin{cases} 
0 \text{ if $m_i \equiv -1 \pmod r$ for any $i\in \{1,\dots,n\}$ or if $D$ is not in $\Z$.}\\
\prod_{e\in E(\Gamma)} |\lgr{e_+}| PD\, \left(\left[\MM_{A_{r-1}}(\tGamma) \right]^{vir}\cap\prod_{i=1}^k \be_{\gamma_i}\right) \text{ otherwise.}\\
\end{cases}
\end{equation}
Here $PD$ denotes the Poincar\'e dual, and $\cap$ denotes the obvious contraction 
$$\cap:\left(H_*(\W(\Gamma),\Q)\otimes \prod_{\tau \in T(\Gamma)}\ch^*_{\ga_{\tau}} \right)\otimes \left(\prod_{\tau \in T(\Gamma)}\ch_{\ga_{\tau}}\right) \rTo H_*(\W(\Gamma),\Q).$$

We will write $\cv_{g,k}(\bm):=\cv_{\Gamma}$ when $\Gamma$ is a genus-$g$ corolla with $k$ tails labeled by $\bm = (m_1,\dots,m_n)$.
\end{df}

\subsubsection{Verification of the $r$-spin axioms}

We will continue to use the notation of Definition~\ref{df:Gammatilde}, and we will write 
\begin{equation}
\MM^{1/r}_{\Gamma}:=\MM_{A_{r-1}}(\tGamma). \glossary{MMr @$\MM^{1/r}_{\Gamma}$ & The stack of $r$-spin curves $\MM_{A_{r-1}}(\tGamma)$}
\end{equation}
This is legitimate, since as discussed earlier, the stack $\MM^{1/r}_{\Gamma}$ of stable $r$-spin curves with graph $\Gamma$  is (canonically) isomorphic to the stack $\MM_{A_{r-1}}(\tGamma)$.

\begin{prop}
The collection of classes $\cv_{\Gamma}$ satisfies all the axioms of an $r$-spin virtual class outlined in \cite{JKV}.
\end{prop}

\begin{proof}
It is clear from the definition that the class $\cv$ lies in $H^{2D}(\MM^{1/r}_{\Ga})$. And, indeed, the only axiom that does not immediately follow from the definition and the corresponding axioms for $A_{r-1}$-curves is the cutting edges axiom.

The axiom requires in Diagram~(\ref{eq:Cutting}) that
\[
q_* pr_2^* \cv_{\Gamma} =  r^{|E(\Gamma)|}
\cv_{ \widehat{\Gamma}} ,\] where $E(\Gamma)$ is the set of edges
of $\Gamma$ that are cut in
 $ \widehat{\Gamma}$.
This follows by induction on the number of edges from the corresponding axiom for the virtual cycle.  For a single edge labeled by $\gamma$, the degree of the map $q$ in that axiom is $|G/\lgr{}|$ and the class $\cv_{\Gamma}$ is $|\lgr{}|$ times the Poincar\'e dual of the virtual cycle, so the overall factor introduced is $|G|^{|E(\Gamma)|} = r^{|E(\Gamma)|}$, as desired.

The only complication is the fact that we have defined the class $\cv$ to be zero if any tail is marked with $r-1$ or $-1$ (corresponding to $\ga = J^0$).  But in the $A_{r-1}$ case the only invariant element in the broad/Ramond ($J^0$-) sector is $0$, so any time a cut graph $\hGamma$ introduces a new tail decorated with $m = r-1$ (or $-1$), the corresponding class will vanish, as required.
\end{proof}

\section{Conclusion}

The cohomological field theory arising from the $A_{r-1}$ theory is given by 
\begin{align*}
\Lambda_{g,n}^{A_{r-1}} (\be_{J^{k_1}},\dots,\be_{J^{k_n}}) :&=\frac{|G|^g}{\deg(st)} PD(st_*([\W_{g,n}]^{vir} \cap \be_{J^{k_1}}\otimes\cdots\otimes\be_{J^{k_n}}))\\
&=\frac{r^g}{r^{2g-1}} PD(st_*([\W_{g,n}]^{vir} \cap  \be_{J^{k_1}}\otimes\cdots\otimes\be_{J^{k_n}}))\\
&=\frac{1}{r^{g-1}} \st_*\cv(k_1-1,\dots,k_n-1)
\end{align*}
And this is precisely the $r$-spin cohomological field theory $\Lambda^{1/r}_{g,n}(\r_{k_1-1}, \dots, \r_{k_n-1})$ as defined in \cite{JKV}, so the two theories are identical.

Faber, Shadrin, and Zvonkine \cite{FSZ} have proved that any $r$-spin cohomological field theory arising from an $r$-spin virtual class is completely determined by the $g=0$ theory, and in \cite{JKV} it is proved that the $g=0$ $r$-spin theory is completely determined by the axioms.  Therefore the proof in \cite{FSZ} that the $r$-spin theory satisfies the Witten Integrable Hierarchies Conjecture also applies to the $A_{r-1}$ theory, as expected.

\bibliographystyle{amsplain}

\providecommand{\bysame}{\leavevmode\hbox
to3em{\hrulefill}\thinspace}

\end{document}